\documentclass{article}
\usepackage[utf8]{inputenc}

\usepackage{amssymb,amsmath,amsthm}
\usepackage{graphicx}
\usepackage{hyperref}

\theoremstyle{plain}
\newtheorem{theorem}{Theorem}[section]

\newtheorem{corollary}[theorem]{Corollary}

\newtheorem{conjecture}[theorem]{Conjecture}
\newtheorem{proposition}[theorem]{Proposition}

\newtheorem{problem}[theorem]{Problem}
\theoremstyle{definition}

\newtheorem{claim}[theorem]{Claim}

\newtheoremstyle{TheoremNum}
        {\topsep}{\topsep}              
        {\itshape}                      
        {}                              
        {\bfseries}                     
        {.}                             
        { }                             
        {\thmname{#1}\thmnote{ \bfseries #3}}
    \theoremstyle{TheoremNum}
    \newtheorem{thmn}{Theorem}

\newcommand\cF{{\mathcal F}}
\newcommand\cG{{\mathcal G}}

\if11     
\usepackage[mathlines]{lineno}
\newcommand*\patchAmsMathEnvironmentForLineno[1]{%
  \expandafter\let\csname old#1\expandafter\endcsname\csname #1\endcsname
  \expandafter\let\csname oldend#1\expandafter\endcsname\csname end#1\endcsname
  \renewenvironment{#1}%
     {\linenomath\csname old#1\endcsname}%
     {\csname oldend#1\endcsname\endlinenomath}}%
\newcommand*\patchBothAmsMathEnvironmentsForLineno[1]{%
  \patchAmsMathEnvironmentForLineno{#1}%
  \patchAmsMathEnvironmentForLineno{#1*}}%
\AtBeginDocument{%
\patchBothAmsMathEnvironmentsForLineno{equation}%
\patchBothAmsMathEnvironmentsForLineno{align}%
\patchBothAmsMathEnvironmentsForLineno{flalign}%
\patchBothAmsMathEnvironmentsForLineno{alignat}%
\patchBothAmsMathEnvironmentsForLineno{gather}%
\patchBothAmsMathEnvironmentsForLineno{multline}%
}

\title{Forbidden subposet problems in the grid}
\author{D\'aniel Gerbner$^1$ \ \ D\'aniel T. Nagy$^1$ \ \ Bal\'azs Patk\'os$^{1,2}$ \ \ M\'at\'e Vizer$^1$ \\ \small $^1$ Alfr\'ed R\'enyi Institute of Mathematics\\ \small $^2$ Moscow Institute of Physics and Technology
}
\date{}

\begin{document}

\maketitle

\begin{abstract}
    For posets $P$ and $Q$, extremal and saturation problems about weak and strong $P$-free subposets of $Q$ have been studied mostly in the case $Q$ is  the Boolean poset $Q_n$, the poset of all subsets of an $n$-element set ordered by inclusion. In this paper, we study some instances of the problem with $Q$ being the grid, and its connections to the Boolean case and to the forbidden submatrix problem.
\end{abstract}

\section{Introduction}
In extremal combinatorics, Tur\'an type problems ask for the largest size that a combinatorial object can have if it is does not contain a prescribed forbidden substructure. Graphs with the most number of edges not containing a fixed subgraph, set systems with the most number of sets not containing two with prescribed intersection size, etc. In this flavor, the forbidden subposet problem for posets $P$ and $Q$ asks for the size of the largest subset of $Q$ that does not contain $P$ as a subposet. There exist two notions of a subposet: we say that $P$ is a \textit{weak subposet} of $R$ if there exists an injection $i:P\rightarrow R$ such that for every $p,p'\in P$, $p\leqslant_P p'$ implies $i(p)\leqslant_R i(p')$. If in addition the injection $i$ satisfies $p\leqslant_P p'$ if and only if $i(p) \leqslant_R i(p')$, then we say that $P$ is a \textit{strong subposet} of $R$. Otherwise we say that $R$ is \textit{weak / strong $P$-free}. Strong subposets are also called \textit{induced} subposets in the literature, while weak subposets are often referred to as subposets, and sometimes as \textit{not necessarily induced} subposets. The extremal numbers $La(Q,P)$ and $La^*(Q,P)$ are defined as the size of the largest weak / strong $P$-free subposet of $Q$. Most of the research, initiated by Katona and Tarj\'an \cite{KT} in the early eighties, focused on the case $Q=Q_n$ the Boolean cube poset $\{0,1\}^n$, i.e., the poset of all subsets of an $n$-element set ordered by inclusion. As it is usual in the literature, we use the notation $La(n,P)$ and $La^*(n,P)$ instead of $La(Q_n,P)$ and $La^*(Q_n,P)$. For a survey on the topic see \cite{GL} and an even more recent summary is Chapter 7 of \cite{GP}. 

One of the most used tools in addressing forbidden subposet problems in the case of the Boolean cube is to find a simpler poset structure $R$ in the cube and apply some averaging argument to the result obtained for $R$. Simpler structures include the chain, the double chain, complementary chain pairs, intervals of a cycle, etc., see \cite{BN,DKS,GP2}. In this note, we focus on the grid $[k]^d=\{1,2,\dots,k\}^d$ ordered coordinate-wise. Elements will be denoted by lower case letters $a,b,x,y,\dots$ etc., and the $i$th ($1 \le i \le d$) coordinate by $a_i,b_i,x_i,y_i, \ldots$ etc. The order ($\leqslant$) on $[k]^d$ is defined as follows: for $x,y \in [k]^d$ we have that  $x\leqslant y$ if and only if $x_i\le y_i$ for all $i=1,2,\dots,d$. Let the \textit{rank} $r(x)$ of an element $x \in [k]^d$ be defined as $\sum_{i=1}^dx_i$. The set of all elements of rank $r$ ($d \le r \le kd$) will be denoted by $S_{k,d,r}$ and we write $s_{k,d,r}=|S_{k,d,r}|$. It is well-known \cite{And} that the size $w_{k,d}$ of the largest antichain in $[k]^d$ is $s_{k,d,\lfloor (k+1)d/2\rfloor}=(1+o(1))\sqrt{\frac{6}{\pi d}} k^{d-1}$ as $\min \{k,d\} \rightarrow \infty$. Therefore, there exist constants $C=C_d$ and $\varepsilon=\varepsilon_{d}$ such that if $|i-\frac{(k+1)d}{2}|<\varepsilon k$, then $s_{k,d,i}\ge Ck^{d-1}$. 

Connection between forbidden subposet problems for the Boolean poset and some extremal problems on the grid was first established by Methuku and P\'alv\"olgyi \cite{MP}. The \textit{dimension} of a poset $P$ is the smallest number $d$ for which there exist $d$ permutations $\pi_1,\pi_2,\dots,\pi_d$ of the elements of $P$ such that $p<_Pq$ if and only if $\pi_i(p)<\pi_i(q)$ for all $i=1,2,\dots,d$. This is clearly equivalent to the fact that $d$ is the smallest integer for which $[|P|]^d$ contains a strong copy of $P$. Methuku and P\'alv\"olgyi showed - by embedding any poset $P$ of the Boolean cube to some grid - that the forbidden subposet problem is naturally connected to the forbidden sub(hyper)matrix problem for the permutation hypermatrix defined by the $\pi_i$s. Applying a Marcus-Tardos-type theorem for hypermatrices \cite{KM,MT}, they proved that for any poset $P$ there exists a constant $c_P$ such that $La^*(n,P)\le c_P\binom{n}{\lfloor n/2\rfloor}$ (the analogous statement for weak subposets follows trivially from a result of Erd\H os \cite{E}).
Their result was strengthened by M\'eroueh \cite{M} and Tomon \cite{T}.

Forbidden subposet problems on the grid and their connection to the case of the Boolean cube was first studied by Tomon \cite{T2} and Sudakov, Tomon, Wagner \cite{STW}. In \cite{STW}, the following general framework was introduced. We say that a formula is \textit{affine}, if it is built from variables, the lattice
operators $\wedge$ and $\vee$ (or to avoid confusion, one might prefer to use $\cup$ and $\cap$), and parentheses (,) (constants are not allowed, e.g.
$x\cap \{1, 2, 3\}$ is not an affine formula). Also, an \textit{affine statement} is a statement of the form
$f < g$ or $f = g$, where $f$ and $g$ are affine formulas. Finally, an \textit{affine configuration} is a Boolean
expression, which uses symbols $\wedge,\vee,\neg$ and whose variables are replaced with affine statements.
Given an affine configuration $C$ with $k$ variables, a lattice $L$
contains $C$, if there exists $k$
distinct elements of $L$ that satisfy $C$, otherwise, say that $L$ avoids $C$. Let $ex(L, C)$ denote the
size of the largest subposet $L'$ of $L$
such that $L'$ avoids $C$. 

Note that in the Boolean cube $\{0,1\}^n$, the lattice operators are $\cup$ and $\cap$ and $<$ is simply $\subsetneq$. Given a poset $P$ with relation $\prec$, the weak and strong $P$-free properties can be described with the following affine configurations:
$$
C_P:=\bigwedge_{p,q\in P, p\prec q}(p<q) \hskip 2truecm C^*_P:=\bigwedge_{p,q\in P, p\prec q}(p<q)\hskip 0.3truecm \wedge \bigwedge_{p,q\in P, p\not\prec q, q\not\prec p}(\neg (p<q) \wedge \neg (q< p)).
$$

\begin{theorem}[Theorem 3.1 in \cite{STW}]\label{svejc}
Let $d$ be a positive integer, $C$ an affine configuration and $c, \alpha > 0$ such that $ex([k]^d, C) \le ck^{d-\alpha}$ holds for
every sufficiently large $k \in \mathbb{N}$. Then we have
$$ex(2^{[n]}, C) \le (1 + o(1))c
\left(\frac{2d}{\pi n}\right)^{\alpha/2}2^n.$$
\end{theorem}

Applying Theorem \ref{svejc} with $\alpha=1$ one obtains the following (see p.17 in \cite{STW}).

\begin{corollary}\label{gridboole} \ 
Let $d$ be a fix natural number.

\smallskip 

(i) If $\limsup_{k \rightarrow \infty} La([k]^d,P)\frac{\sqrt{d}}{k^{d-1}}\le c$, then $\limsup_{n \rightarrow \infty} \frac{La(n,P)}{\binom{n}{\lfloor n/2\rfloor}}\le c$ holds.

(ii) If $\limsup_{k \rightarrow \infty} La^*([k]^d,P)\frac{\sqrt{d}}{k^{d-1}}\le c$, then $\limsup_{n \rightarrow \infty} \frac{La^*(n,P)}{\binom{n}{\lfloor n/2\rfloor}}\le c$ holds.
\end{corollary}

We are going to prove a theorem that is similar in flavour to Theorem \ref{svejc}. 
However the proof of Theorem~\ref{vissza} is much simpler than that of Theorem \ref{svejc}, as the authors of \cite{STW} applied an involved chain partition theorem to obtain their result, while for us a relatively simple averaging argument would suffice.

\begin{theorem}\label{vissza}
Let $d$ be a positive integer, $C$ an affine configuration, and $c, \alpha > 0$ such that $ex([k]^d, C) \le ck^{1-\alpha}\cdot w_{k,d}$ holds for
every sufficiently large $k \in \mathbb{N}$. Then we have
$$ex(2^{[n]}, C) \le (1 + o(1))cn^{1-\alpha}
\binom{n}{\lfloor n/2\rfloor}.$$
\end{theorem}

Observe that the following strengthening of Corollary \ref{gridboole} is an immediate consequence of Theorem \ref{vissza} with $\alpha=1$. 

\begin{corollary}\label{gridboole2} \ 
Let $w_{k,d}$ denote the size of the largest antichain in $[k]^d$.

(i) If $\limsup_{k \rightarrow \infty} \frac{La([k]^d,P)}{w_{k,d}}\le c$, then $\limsup_{n \rightarrow \infty} \frac{La(n,P)}{\binom{n}{\lfloor n/2\rfloor}}\le c$ holds.

(ii) If $\limsup_{k \rightarrow \infty} \frac{La^*([k]^d,P)}{w_{k,d}}\le c$, then $\limsup_{n \rightarrow \infty} \frac{La^*(n,P)}{\binom{n}{\lfloor n/2\rfloor}}\le c$ holds.
\end{corollary}

Also, one can compare the two theorems: Theorem \ref{svejc} is stronger as long as $\alpha< 1$ (moreover, Theorem~\ref{vissza} is meaningless if $\alpha<1/2$), while Theorem \ref{vissza} is stronger for $\alpha>1$ and they are exactly of the same strength (apart from a multiplicative constant factor $\sqrt{\frac{6}{\pi}}$) if $\alpha=1$, i.e., the case of Corollary \ref{gridboole} and \ref{gridboole2}.

\vskip 0.2truecm

Next we start investigating the forbidden subposet problem on the grid for specific posets. This topic is naturally connected to the area of forbidden sub(hyper)matrices. We say that a $0$-$1$ matrix $A$ \textit{contains} another $0$-$1$ matrix $M$ if $A$ has a submatrix $M'$ that can be turned into $M$ by replacing some (possibly zero) 1 entries with 0. Otherwise we say that $A$ \textit{avoids} $M$. For $k,l \ge 1$ and a $0$-$1$ matrix $M$ let us define $ex(k,l,M)$ as the largest number of $1$-entries in a $k\times l$ $0$-$1$ matrix $A$ that avoids $M$. (In the literature, $P$ is used for $M$ to hint at the word \textit{pattern} for the avoided matrix $M$, but since we use $P$ for posets, to avoid confusion $M$ stands for the forbidden matrix in this paper.) An overview of results on this extremal function can be found in the introduction of \cite{FK}. Now consider a poset $P$ of dimension 2 and all possible embeddings of $P$ into $[|P|]^2$. Every such embedding $f$ naturally corresponds to a $|P|\times |P|$ $0$-$1$ matrix $T$: the entry $t_{i,j}$ is 1 if and only if $f(p)=(i,j)$ for some $p\in P$. Let $M_f$ denote the submatrix of $T$ of all rows and columns that contain at least one 1-entry. A subposet $Q$ of $[k]^2$ again naturally corresponds to the $k\times k$ $0$-$1$ matrix $M_Q$: its $(i,j)$-entry is 1 if and only if $(i,j)\in Q$. Then clearly,  $Q$ is strong $P$-free if and only if $M_Q$ avoids $M_f$ for all embeddings $f$. This means that every strong forbidden subposet problem on $[k]^2$ corresponds to forbidding several submatrices. The same holds for weak forbidden subposet problems, as containing a weak copy of $P$ is equivalent to containing a strong copy from the family of posets $P'$ that have the same number of elements as $P$ and that contain a weak copy of $P$.

A much studied \cite{BC,FK} permutation pattern is the $s\times s$ matrix $J_s$ that is obtained from the identity matrix by moving its last column to the beginning. Let $\vee_s$ denote the poset on $s+1$ elements $a,b_1,b_2,\dots, b_s$ with $a<b_1,b_2,\dots,b_s$. A copy of $J_s$ in an $n\times m$ matrix corresponds to one possible embedding of $\vee_{s-1}$ into
$[n]\times [m]$. However, we conjecture that the extremal and saturation numbers for $J_s$ correspond to strong extremal and saturation numbers of $\vee_s$. (See Conjecture \ref{satconj}.)

 Set $k_2=c_2=1$ and for $s\ge 3$ let us define $k_s$ to be the maximum $k$ such that $\sum_{j=2}^k j < s$. Finally, let $c_s=s-1-\sum_{j=2}^{k_s}j$.

\begin{theorem}\label{grid} \

(i) For any $s \ge 1$, we have $La([k]^2,\vee_s)=(k_s+\frac{c_s}{k_s+1}+o(1))k$.

(ii) For $k,l \ge 1$ we have  $La^*([k]\times[l],\vee_2)=k+l-1$. 

(iii) For $k,l \ge 2$ we have $La^*([k] \times [l],\vee_3)=2(k+l)-4$. 

(iv) $La^*([k]^2,\vee_s)\ge 2(s-1)k-O(s^2)$.

(v) $La^*([k]\times [l],\{\vee_2,\wedge_2\})=k+l-1$.
\end{theorem}

For a finite poset $P$, let $h(P)$ denote its \textit{height}, i.e., the number of elements of its largest complete subposet. Let $D_k$ denote the $k$-diamond, a poset on $k+2$ elements $a<b_1,b_2,\dots,b_k,c$ with $a<b_i<c$ for all $1\le i \le k$.

\begin{proposition}\label{doublechain}
For any poset $P$, we have
$La([k]^2,P)\le (\frac{|P|+h(P)}{2}-1)k+O(1)$.

In particular,
$La([k]^2,D_2)=\frac{5}{2}k + O(1)$ and $La([k]^2,D_3)=3k+ O(1)$.
\end{proposition}

Every extremal problem has its saturation counterpart. Very recently, there has been an increased attention \cite{Be,BC,FK,G} to the saturation version of the forbidden submatrix problem. Let $sat(n,m,M)$ denote the minimum number of 1-entries of an $n\times m$ binary matrix $A$ that avoids $M$, such that any matrix $A'$ that is obtained from $A$ by changing a 0 to a 1, contains $M$. Fulek and Keszegh \cite{FK} proved that for any matrix $M$ one has $sat(n,n,M)=O(1)$ or $sat(n,n,M)=\Theta(n)$ and asked for a characterization of matrices with linear saturation number. Partial answers were given in $\cite{Be,G}$. 

We start investigating the saturation version of the forbidden subposet problem on the grid. We say that a subposet $Q'$ of $Q$ is \textit{weak/strong $P$-saturated} if it is weak/strong $P$-free, but adding any element of $Q\setminus Q'$ to $Q'$ creates a weak/strong copy of $P$. Let $sat(Q,P)$ and $sat^*(Q,P)$ denote the minimum size of a weak / strong $P$-saturated subposet of $Q$, respectively. Just as in the extremal case, the poset saturation problems on the grid are equivalent to matrix saturation problems with a family of matrices to be avoided. First we observe that in any dimension, the weak poset saturation number is \textit{always} bounded by a constant.

\begin{proposition}\label{satnni}
For any positive integers $p,d\ge 2$ and for any $p$-element poset $P$ and integer $k$ we have $sat([k]^d,P)\le \sum_{r=d}^{d+p-2}s_{k,d,r}$.
\end{proposition}

Note that $\sum_{r=d}^{d+p-2}s_{k,d,r}$ does not change once $k\ge p$ and thus the upper bound is a constant independent of $k$.

Let us remark that the proof of Proposition \ref{satnni} stays valid for $[k_1] \times [k_2] \times \ldots \times [k_n]$, if we replace  $\sum_{r=d}^{d+p-2}s_{k,d,r}$ with the size of the $p-1$ lowest levels.

\medskip 

Based on the above mentioned result of Fulek and Keszegh, we show an analogous theorem for the strong saturation number of posets.

\begin{theorem}\label{strongsat}
For any poset $P$ with $dim(P)=2$ we either have $sat^*([k]^2,P)=O(1)$ or $sat^*([k]^2,P)=\Theta(k)$.
\end{theorem}

Finally, we address the saturation problem for some specific posets.

\begin{theorem} \label{satstar} \

(i)
For any $k,l$ we have $sat^*([k]\times [l], \{\vee_2,\wedge_2\})=\max\{k,l\}$.

(ii)
For any integers $k,l$ we have $sat^*([k]\times [l],\vee_2)=La^*([k]\times [l],\vee_2)=k+l-1$.

\end{theorem}

\begin{theorem} \label{satvees}
Let $P$ be a poset with $dim(P)=2$ such that a strong copy of $P$ in a two dimensional grid cannot contain two neighboring points. Then $sat^*([k]\times [l],P)\ge \max\{k,l\}$.
\end{theorem}

Note that a poset, that satisfies the assumption of Theorem \ref{satvees} is any $\vee_s$ with $s\ge 3$.

\section{Connection to the hypercube - the proof of Theorem \ref{vissza}}


Let us start with stating the theorem again. Recall that $w_{k,d}=s_{k,d,\lfloor (k+1)d/2\rfloor}$ is the size of the largest antichain in $[k]^d$. 

\begin{thmn}[\ref{vissza}]
Let $d$ be a positive integer, $C$ an affine configuration, and $c, \alpha > 0$ such that $ex([k]^d, C) \le ck^{1-\alpha}\cdot w_{k,d}$ holds for
every sufficiently large $k \in \mathbb{N}$. Then we have
$$ex(2^{[n]}, C) \le (1 + o(1))cn^{1-\alpha}
\binom{n}{\lfloor n/2\rfloor}.$$
\end{thmn}

Let us mention that our proof is similar to that of Methuku and P\'alv\"olgyi \cite{MP}. The main difference is that we only consider partitions with equal (or almost equal) parts and that we optimize our calculations in order to obtain the best possible constants.

\begin{proof}[Proof of Theorem \ref{vissza}]
Let $\cF\subset 2^{[n]}$ be a family that avoids $C$. Let us write $n$ in the following form: $n=d(k-1)+r$ with $0\le r<d$ and set $n'=n-r$. We partition $\cF$ into $2^r$ subfamilies $\cF_A$ indexed with subsets of $[n]\setminus [n']$ such that $\cF_A=\{F\setminus A: F\in \cF, F\cap ([n]\setminus [n'])=A\}$. If we can prove that for every $\cF_A$ we have $|\cF_A|\le (1+o(1))cn^{1-\alpha}\binom{n'}{\lfloor \frac{n'}{2}\rfloor}$, then $|\cF|=\sum_{A}|\cF_A|\le 2^r(1+o(1))cn^{1-\alpha}\binom{n-r}{\lfloor \frac{n-r}{2}\rfloor}=(1+o(1))cn^{1-\alpha}\binom{n}{\lfloor \frac{n}{2}\rfloor}$.
 
We say that a subset $F$ of $[n']$ is a \textit{$\pi$-block} for a permutation $\pi$ of $[n']$ if for every $0\le j\le d-1$ the intersection $F\cap \{\pi(j(k-1)+1),\pi(j(k-1)+2),\dots,\pi((j+1)(k-1))\}$ is an initial segment, i.e., $\{\pi(j(k-1)+1),\pi(j(k-1)+2),\dots,\pi(j(k-1)+h)\}$ for some $h=0,1,2,\dots,k-1$ (here $h=0$ means that the intersection is empty). Observe that union and intersection of any pair of $\pi$-blocks is a $\pi$-block again. More importantly, the set of $\pi$-blocks is isomorphic to $[k]^d$. A $\pi$-block $F$ is identified with the $d$-tuple of the size of the intersection of $F$ and the initial segments - all sizes increased by one.

Let $\cG\subseteq 2^{[n']}$ be a family that avoids an affine configuration $C$. We count the pairs $(F,\pi)$ such that $F\in \cG$ is a $\pi$-block. For a fixed permutation $\pi$, the number of pairs is clearly at most $ex([k]^d,C)$, so the total number of pairs is at most $n'! \cdot ex([k]^d,C)$, which is, by the assumption, at most $(c+o(1))n'^{1-\alpha}w_{k,d} \cdot n'!$.

On the other hand, for any set $F$, the permutations $\pi$ for which $F$ is a $\pi$-block are in bijection with the set $\{(\alpha,\beta, (f_1,f_2,\dots,f_d))\}$, where $\alpha$ is an order of the elements of $F$, $\beta$ is an order of elements of $[n']\setminus F$ and $(f_1+1,f_2+1,\dots,f_d+1)$ satisfies $\sum_{i=1}^df_i=|F|$, $f_i\le k-1$ for all $i=1,2,\dots,d$. Indeed, the $f_i$s tell us how large the $i$th initial segment of $F$ is in $\pi$. Therefore, the number of such permutations for a fixed $F$ is $ s_{k,d,|F|+d}|F|!(n'-|F|)!$. We obtained
\begin{equation}\label{lym}
\sum_{F\in \cG}s_{k,d,|F|+d}|F|!(n'-|F|)!\le (c+o(1))n'^{1-\alpha}w_{k,d} \cdot n'!.
\end{equation}

\begin{claim}\label{kisclaim}
We have $s_{k,d,j+d} \cdot j!(n'-j)!\ge (1-o(1))w_{k,d}\lfloor \frac{n'}{2}\rfloor!\lceil \frac{n'}{2}\rceil!$.
\end{claim}

\begin{proof}[Proof of Claim]
First of all, by symmetry, it is enough to prove the statement for $j<\frac{(k-1)d}{2}$. On the one hand, we have $\frac{(\lfloor n'/2\rfloor-i)!(\lceil n'/2\rceil+i)!}{\lfloor n'/2\rfloor!\lceil n'/2\rceil!}=\prod_{j=1}^i\frac{\lceil n'/2\rceil +i-j+1}{\lfloor n'/2\rfloor -j+1}\ge (1+\frac{2i}{n'})^i$.
 
On the other hand, we have $s_{k,d,i}=s_{k,d,i-1}+s_{k,d-1,i-1}-s_{k,d-1,i-k}$. Indeed, the number of elements of $[k]^d$ with rank $i$ and first coordinate 1 is $s_{k,d-1,i-1}$. The elements with rank $i$ and non-one first coordinate are in 1-to-1 relationship with those of rank $i-1$ and first coordinate \textit{not} $k$, so their number is $s_{k,d,i-1}-s_{k,d-1,i-k}$ (the bijection is established by subtracting 1 from the first coordinate). 
Rearranging and omitting a negative factor from the right hand side, we obtain $s_{k,d,i}-s_{k,d-1,i-1}\le s_{k,d,i-1}$.
 
If $\varepsilon>0$ is small enough, then for $i\ge n'\frac{1-\varepsilon}{2}$ we have that $s_{k,d,i}=\Theta(k^{d-1})$ and $s_{k,d-1,i-1}=\Theta(k^{d-2})$ as pointed out in the introduction. Therefore, $\frac{s_{k,d,i-1}}{s_{k,d,i}}\ge 1-\frac{s_{k,d-1,i-1}}{s_{k,d,i}}\ge 1-\frac{c_0}{n'}$ for some absolute constant $c_0$.

Putting this together: if $j\ge n'\frac{1-\varepsilon}{2}$, then writing $j=\lfloor n'/2\rfloor -j'$
\[ \frac{j!(n'-j)!s_{k,d,j+d}}{\lfloor n'/2\rfloor !\lceil n'/2\rceil!s_{k,d,\lfloor n'/2\rfloor+d}}\ge \left[\left(1+\frac{2j'}{n'}\right)\left(1-\frac{c_0}{n'}\right)\right]^{j'}.
\]
The right hand side is greater than 1, if $j$ is at least than $n'/2-c_1$ for some absolute constant $c_1$ and always $1-o(1)$.
 
Finally, if $j<n'\frac{1-\varepsilon}{2}$, then $j!(n'-j)!$ is already larger than $w_{k,d}\lfloor n'/2\rfloor !\lceil n'/2\rceil!$.
\end{proof}

By Claim \ref{kisclaim} and (\ref{lym}) we have 
\[
|\cG|(1-o(1))w_{k,d}\left\lfloor \frac{n'}{2}\right\rfloor!\left\lceil \frac{n'}{2}\right\rceil! \le (c+o(1))n'^{1-\alpha}w_{k,d}n'!.
\]
After rearranging and using the first paragraph of this proof, the statement of the theorem follows.
\end{proof}

\section{Results for the grid}

Recall that $k_2=c_2=1$, for $s\ge 3$ we defined $k_s$ to be the maximum $k$ such that $\sum_{j=2}^k j < s$, and $c_s=s-1-\sum_{j=2}^{k_s}j$. Theorem \ref{grid} (i) states that
for any $s \ge 1$, we have $La([k]^2,\vee_s)=(k_s+\frac{c_s}{k_s+1}+o(1))k$.

\begin{proof}[Proof of Theorem \ref{grid} (i)]
For the lower bound, we need a construction. For every $1\le i\le k$, let the leftmost element of $F$ in the $i$th row be the element in the diagonal. Furthermore, in each row, let the element of $F$ form an interval of size $k_s$ or $k_s+1$. If $i\equiv 1,2,\dots,c_s \mod k_s+1$, then let the length be $k_s+1$, otherwise $k_s$. (See Figure \ref{constfigure}.) The minimal elements of $F$ are exactly those in the diagonal, and the number of other elements greater than one such element is at most $\sum_{j=2}^{k_s}j+c_s=s-1$. The size of $F$ is $(k_s+\frac{c_s}{k_s+1})k-O(s)$.

\begin{figure}[!htp]
		\centering
		\includegraphics[width=0.6\linewidth]{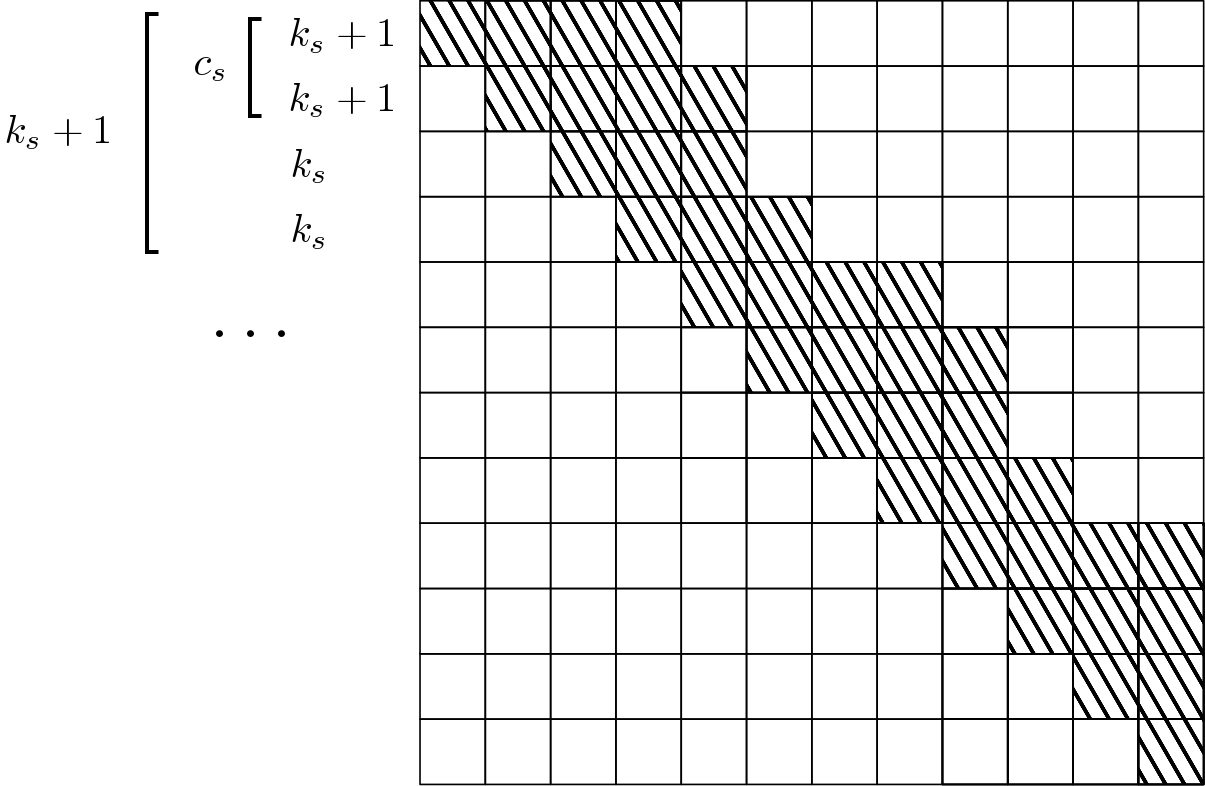}
		\caption{The construction given for Theorem \ref{grid} (i). In this example $k=12$ and $s=8$, therefore $k_s=3$ and $c_s=2$.}
		\label{constfigure}
\end{figure}

 Let $F\subseteq [k]^2$ be a $\vee_s$-free set of elements. Let $M\subseteq F$ denote the subset of minimal elements. Observe that elements of $M$ are both leftmost in their row, and lowest in their column. Also, to check whether $F$ is $\vee_s$-free, it is enough to check whether any element of $M$ is smaller than less than $s$ other elements of $F$. Because of this, we can assume that $F$ is convex, i.e., for any $f,f',g\in [k]^2$ with $f,f'\in F$, $f<g<f'$, we have $g\in F$. Indeed, if $g\notin F$, then we can replace $f'$ by $g$, and as minimal elements remain the same, we preserve the $\vee_s$-free property.  This implies that we can assume that in any row and column, the elements of $F$ form an interval, and if $m_i$ denotes the second coordinate of the smallest element of $F$ in the $i$th column, then the $m_i$s form a non-increasing sequence.
 
 Let $b_i$ denote the number of elements of $F$ in the $i$th row, so $\sum_{i=1}^k b_i=|F|$. Let $M'$ denote the set containing the lowest element of all columns. By definition, $M\subseteq M'$. For any pair $(m,f)$ with $m\in M',f\in F$, $m \leqslant f$, we can appoint the pair $(f',f)$, where $f'$ is the element at the intersection of the column of $m$ and the row of $f$. By the convexity of $F$, we must have $f'\in F$. Also, as $m$ is the lowest element of its column, this is a bijection. Therefore, we obtained that the number of such pairs is $\sum_{i=1}^k (\binom{b_i}2+b_i)=\sum_{i=1}^k\binom{b_i+1}{2}$. This function is convex, thus for fixed $|F|$, its minimum is attained when the $b_i$s differ by at most 1. Therefore, if $|F|>(k_s+\frac{c_s}{k_s+1}+o(1))k$, then the number of pairs $(f,b)$ with $f\in M'$, $b\in F$, $f\leqslant b$ is more than $(s-1)k$, so there must exist an $m\in M'$ smaller than at least $s$ many other elements.
\end{proof}

Let us continue with Theorem \ref{grid} (iii), which states that for $k,l \ge 2$ we have $La^*([k] \times [l],\vee_3)=2(k+l)-4$.

\begin{proof}[Proof of Theorem \ref{grid} (iii)] To see the lower bound let us consider the set $\{(a,b) : a = 1 \textrm{ or }l \textrm{ or } b = 1 \textrm{ or }k\}$. 

We prove the upper bound by contradiction. Let us denote the elements of copies of $\vee_3$ by $A,B,C$ and $D$: let $A$ be the smallest element and the other elements $B,C$ and $D$ ordered  by their second coordinate decreasingly (equivalently, by their first coordinate increasingly). If $k=2$ or $l=2$, then the statement is trivially true. Let us consider a counterexample $F\subseteq [k]\times[l]$ with $k+l$ minimal and among these the sum of the coordinates of the elements is maximal. If $F$ contains at most 2 elements in a row or in a column, then we can delete them, and either we get a counterexample with smaller $k+l$ or either $k$ or $l$ is 2. Therefore all rows and columns contain at least 3 elements of $F$.

The next observation is that one can put $(k,l)$ into $F$ without violating the strong $\vee_3$-freeness condition. 

If $(1,1)\in F$, then $F$ contains at most two elements from each diagonal, so we are done.

If $(1,1) \not \in F$, then let $(k_1,1)$ be the first element in $F$ in the first row (so $k_1 \ge 2$). As we have at least two elements above $(k_1,1)$ and at least two elements right to $(k_1,1)$, we have that $(k_1+1,2) \not \in F$ by the strong $\vee_3$-free property. Note that by the maximality of the sum of the coordinates in $F$ we cannot replace $(k_1,1)$ by $(k_1+1,2)$ in $F$. This means that $F\setminus \{(k_1,1)\}\cup \{(k_1+1,2)\}$ contains a strong $\vee_3$. Note also that $(k_1+1,2)$ can only play the role of $D$ in this strong $\vee_3$. The role of $A$ can be played only by an element $(k',2)$ with $k' < k_1$. Let $(k_2,2)$ be the first element in the second row. Similar way as above we have that $(k_2+1,3) \not \in F$ and we cannot put $(k_2,2)$ into $(k_2+1,3)$ by the maximality of the sum of the coordinates of the elements in $F$, so $(k_2+1,3)$ would create a strong $\vee_3$. In that strong $\vee_3$, the element $(k_2+1,3)$ can only play the role of $D$:

\medskip 

$\bullet$ it can not be $A$, as otherwise $(k_2,2)$ could also play the role of $A$ instead,

\smallskip 

$\bullet$ it can not be a $B$ as there is at least 1 element above $(k_2+1,3)$ in the $(k_2+1)$th column, and that could play the role of $B$ as well,

\smallskip 

$\bullet$ it can not be $C$ as there is no element $(x,y) \in F\setminus\{(k_2,2)\}$ with $x \le k_2$ and $y \le 2$.

\medskip 

This implies that there is $(k'',3) \in F$ with $k''< k_2$, and let $(k_3,3)$ be the first element in $F$ in the third row.
We can continue the same way: we have elements $(k_1,1), (k_2,2), (k_3,3), ...$ with $k_1 > k_2 > k_3 >...$. If the first element in the first column is $(1,t)$, then we have $k_1 \ge t$. By changing the role of columns and rows in the above reasoning, we obtain $k_1 \le t$, and so $k_1=t$. The previous argument also implies that $k_j=k_1+1-j$ for $1 \le j \le k_1$; $\{(k_1,1), (k_2,2), (k_3,3), ..., (1,k_1)\}=:M\subset  F$ and $\{(k_1+1,2), (k_2+1,3), (k_3+1,4), ..., (2,k_1+1)\}=:M^+ $ is disjoint with $F$. (Observe that $k_1< \min\{k,l\}$ as every row and column contains at least 2 elements. Therefore there exist $k_1+1$st column and row.) 

Let us consider $F' := F \setminus M \cup M^+$. Note that we described the possible strong copies of $\vee_3$ in $F$. By that it is easy to see that $F'$ is also strong $\vee_3$-free and the sum of the coordinates is larger, a contradiction.
\end{proof} 

Recall that Theorem \ref{grid} (iv) states that $La^*([k]^2,\vee_s)\ge 2(s-1)k-O(s^2)$.

\begin{proof}[Proof of Theorem \ref{grid} (iv)] To prove this part we will show a family $G\subseteq [k]^2$ avoiding  strong $\vee_s$. Let $G$ be the union of the $(s-1)$ highest rows and the $(s-1)$ rightmost columns. Then $|G|=2(s-1)k-(s-1)^2$, and $G$ is strong $\vee_s$-free, since none of its elements can be the minimal element of a strong $\vee_s$.
\end{proof}

We will prove Theorem \ref{grid} (ii) and (v) later, together with the corresponding saturation statements from Theorem \ref{satstar}.
Let us turn to  Proposition \ref{doublechain}, which states that for any poset $P$, we have
$La([k]^2,P)\le (\frac{|P|+h(P)}{2}-1)k+O(1)$. Moreover, matching lower bounds are given in case of the posets $D_2$ and $D_3$.

\begin{proof}[Proof of Proposition \ref{doublechain}]
In \cite{BN}, a poset called the {\it infinite double chain} is introduced. Its elements are $L_i,~M_i,~i\in\mathbb{Z}$. The defining relations between the elements are
$i<j\Rightarrow L_i< L_j,~ L_i< M_j,~ M_i< L_j$. Burcsi and Nagy proved that if a subset of the infinite double chain is $P$-free for some finite poset $P$, then its size is at most $|P|+h(P)-2$.

Now let $F$ be a $P$-free subset of $[k]^2$. We call the points $(x,y)\in [k]^2$ for which $x-y$ is constant an increasing diagonal. Consider the union of four consecutive increasing diagonals. This structure is isomorphic to a subset of the infinite double chain: the middle two diagonals correspond to elements $L_i$ while the outer elements correspond to the elements $M_i$. Therefore there can be at most $|P|+h(P)-2$ elements of $F$ there. There are a total of $2k-1$ increasing diagonals, so we can partition them into $\left\lceil\frac{k}{2}\right\rceil-1$ groups, leaving out the last one or three diagonals, a constant number of elements. In conclusion, $|F|\le (\left\lceil\frac{k}{2}\right\rceil-1)(|P|+h(P)-2)+O(1)\le(\frac{|P|+h(P)}{2}-1)k+O(1)$.

The upper bounds for $D_2$ and $D_3$ follow from the general result proved above. For the lower bounds, consider the following $D_2$-free and $D_3$-free families in $[k]^2$:
$$\{(x,y)\in[k]^2~:~x+y\in\{k,~k+2\}\}\cup \{(x,y)~:~x+y=k+1~\text{and}~x~\text{is odd}\},$$
$$\{(x,y)\in[k]^2~:~k\le x+y\le k+2\}.$$

\end{proof}

Let us turn to saturation problems. Recall that Proposition \ref{satnni} states that the weak saturation number is always upper bounded by a constant that does not depend on $k$. More precisely, for any positive integers $p,d\ge 2$ and for any $p$-element poset $P$ and integer $k$ we have $sat([k]^d,P)\le \sum_{r=d}^{d+p-2}s_{k,d,r}$.

\begin{proof}[Proof of Proposition \ref{satnni}]
For any enumeration $\pi=x_1,x_2,\dots x_{k^d}$ of $[k]^d$, one can create the $P$-saturating family $\cF_\pi\subseteq [k]^d$ with respect to $\pi$ greedily as follows: we let $\cF_0=\emptyset$ and whenever $\cF_{i-1}$ is defined, we set $\cF_i=\cF_{i-1}\cup \{x_i\}$ if $\cF_{i-1}\cup \{x_i\}$ is $P$-free, and let $\cF_i=\cF_{i-1}$ otherwise. By definition, $\cF_\pi:=\cF_{k^d}$ is $P$-saturating.

Let $\pi=x_1,x_2,\dots,x_{k^d}$ be an enumeration of $[k]^d$ such $r(x_i)\le r(x_j)$ for any $i\le j$. We claim that $\cF_\pi$ is downward closed, i.e., if $x_i\in \cF_\pi$, then any $y\le x_i$ belongs to $\cF_\pi$. Indeed, the property of the enumeration $\pi$ ensures that at any moment when we decide about whether to include an $x_j$, then any element of $\cF_{j-1}$ that is in relation with $x_j$ must be smaller than $x_j$. Also, any $y\le x_i$ is enumerated before $x_i$, so $y=x_j$ for some $j<i$. If $y=x_j\notin \cF_\pi$, then it is because $\cF_{j-1}\cup \{y\}$ contains a copy of $P$ that contains $y$, thus $y$ is a maximal element in that copy of $P$. But by the above, if we replace $y$ by $x_i$, then we get another copy of $P$. (Here we use that we look for a weak copy of $P$.)

Clearly, $\cF_{\pi}$ cannot contain a chain of length $p$ as that is a weak copy of $P$. Therefore, we must have $\cF_\pi\subseteq \cup_{r=d}^{d+p-2}S_{d,k,r}$ and the result follows.
\end{proof}

Observe that enumerations considering low-rank elements first are not necessarily the best even among greedily picked $P$-saturating families. Indeed, if $P$ is the chain of length $p$, then one is much better off considering low and high ranked elements alternatingly. More formally, we say that an enumeration $\pi=x_1,x_2,\dots,x_{k^d}$ is \textit{middle comes last} (MCL) if for any $i\le j$ we have $|r(x_i)-\frac{d(k+1)}{2}|\ge |r(x_j)-\frac{d(k+1)}{2}|$. If $p=2m+1$, then the greedy $C_l$-saturating family with respect to an MCL enumeration is $\cup_{r=d}^{d+m-1}S_{k,d,r}\cup \cup_{r=kd-m+1}^{kd}S_{k,d,r}$, while the enumerations used in Theorem \ref{satnni} yield $\cup_{r=d}^{d+2m-1}S_{k,d,r}$ which is significantly larger. We do not know whether MCL enumerations always give the best greedy approach..

\smallskip

Let us continue with strong saturation. Recall that Theorem \ref{strongsat} states that for any poset $P$ with $dim(P)=2$ we either have $sat^*([k]^2,P)=O(1)$ or $sat^*([k]^2,P)=\Theta(k)$.

\begin{proof}[Proof of Theorem \ref{strongsat}]
Recall that the strong saturation problem for a poset $P$ is equivalent to the saturation problem for a finite set of $0$-$1$ matrices, one of which is a permutation matrix. A result of Marcus and Tardos \cite{MT} states that $ex(n,n,M)=O(n)$ for any permutation matrix $M$, which implies that $sat^*([k]^2,P)=O(k)$ holds.

Fulek and Keszegh \cite{FK} showed that for any matrix $M$, if $sat(n,n,M)$ is not constant, then it is at least linear. Their argument stays valid for any finite set $M_1,M_2,\dots,M_r$ of matrices, so we just sketch it here. If $M_i$ is a $q_i\times p_i$ matrix, then let $q:=\max\{q_i,p_i: i=1,2,\dots,r\}$. If $sat(n,n,\{M_1,M_2\dots,M_r\})\ge \frac{n}{q}$ for every large enough $n$, then clearly the saturation number grows at least linearly. Otherwise, there exist a large enough $n_0$ and an $n_0\times n_0$ matrix $A$ that is $\{M_1,M_2,\dots,M_r\}$-saturated and $A$ contains less than $\frac{n_0}{q}$ 1-entries. Then there must exist $q$ consecutive all-0 rows and $q$ consecutive all-0 columns of $A$. It is easy to check that for any $n>n_0$ if we add $n-n_0$ all-0 rows and columns to $A$ such that together with the $q$ consecutive all-0 rows and columns of $A$ they stay consecutive, then the obtained matrix $A'$ is $\{M_1,M_2,\dots,M_r\}$-saturated and contains the same number of 1-entries as $A$.

This shows that for any finite  set of 0-1 matrices, the saturation number is either constant or at least linear. As the strong saturation problem for poset $P$ is equivalent to the saturation problem for a finite set of 0-1 matrices, this concludes the proof of Theorem \ref{strongsat}.
\end{proof}

It is more convenient for us to prove Theorem \ref{satstar} together with the corresponding results from Theorem \ref{grid}.
Theorem \ref{satstar} (i) states that for any $k,l$ we have $sat^*([k]\times [l], \{\vee_2,\wedge_2\})=\max\{k,l\}$ and Theorem \ref{grid} (v) states that $La^*([k]\times [l],\{\vee_2,\wedge_2\})=k+l-1$. We will use the following notions in the proof.
The \textit{comparability graph} of a poset $P$ has vertex set $P$ and $p\neq q$ are joined by an edge if $p<q$ or $q<p$. A poset is \textit{connected} if its comparability graph is connected, and a component of $P$ is a connected component of its comparability graph.

\begin{proof}[Proof of Theorem \ref{satstar} (i) and Theorem \ref{grid} (v)]
Observe that $P$ is strong $\{\vee_2,\wedge_2\}$-free if and only if the components of $P$ are chains (or equivalently the components of its comparability graph are cliques). Let $F\subseteq [k]\times [k]$ be a strong $\{\vee_2,\wedge_2\}$-saturated set of elements, and let $C_1,C_2,\dots,C_h$ be the components of $F$ (thus we know that each $C_i$ is a chain). For $i=1,2,\dots,h$ let $(a_i,b_i)$ be the minimal element of $C_i$ and $(x_i,y_i)$ be the maximal element of $C_i$. 
\begin{itemize}
\item
As the chains are incomparable (they are the connected components of $F$), after renumbering the chains we can assume that for any $i< j$ we have $x_i<a_j$ and $y_j<b_i$.
\item
As $F$ is saturated, for every $i$ the component $C_i$ is a maximal chain between $(a_i,b_i)$ and $(x_i,y_i)$, as otherwise we could extend $C_i$ to such a chain that is still incomparable with the other chain components of $F$, thus the resulting larger set of elements would be $\{\vee_2,\wedge_2\}$-free, contradicting our assumption on $F$. In particular, $|C_i|=(x_i-a_i+1)+(y_i-b_i+1)-1$.
\item 
For every $1\le \alpha \le k$, there exists $i$ with $a_i\le \alpha \le x_i$. Indeed, if not then there would exist a counterexample $\alpha=x_i+1$ for some $i$. But then adding $(\alpha,y_i)$ to $C_i$ and to $F$ would keep the $\{\vee_2,\wedge_2\}$-free property contradicting the maximality of $F$. Similarly, for every $1\le \beta \le l$, there exists $j$ with $b_j\le \beta \le y_j$. In particular, $\sum_{i=1}^h(x_i-a_i+1)=k$ and $\sum_{i=1}^h(y_i-b_i+1)=l$.
\end{itemize}
The above bullet points yield that $|F|=\sum_{i=1}^h|C_i|=k+l-h$. Thus the size of $F$ is largest if $h=1$ and thus $La^*([k]\times [l],\{\vee_2,\wedge_2\})=k+l-1$, while the size of $F$ is smallest if $h$ is as large as possible. Clearly, $h$ cannot be more than the width of $[k]\times [l]$, which is $\min\{k,l\}$, so $sat^*([k]\times [l],\{\vee_2,\wedge_2\})\ge \max\{k,l\}$ and the construction $\{(k,1),(k-1,2),\dots, (1,k), (1,k+1),\dots,(1,l)\}$ shows that equality holds.
\end{proof}

We continue with Theorem \ref{satstar}
(ii) and Theorem \ref{grid} (ii), which state that for any integers $k,l$ we have $sat^*([k]\times [l],\vee_2)=La^*([k]\times [l],\vee_2)=k+l-1$.

\begin{proof}[Proof of Theorem \ref{satstar} (ii) and Theorem \ref{grid} (ii)]
We proceed by induction on $k+l$ with the base cases $k=1$ or $l=1$ being trivial. Let $F$ be any strong $\vee_2$-saturated subset of $[k]\times [l]$ and observe that $(k,l)\in F$ as it is not contained in any strong copy of $\vee_2$. Let $(a_1,b_1),(a_2,b_2),\dots,(a_h,b_h)$ be the maximal elements of $F\setminus \{(k,l)\}$. By reordering, we may assume  $a_1<a_2<\dots <a_h$ and $b_1>b_2>\dots >b_h$. Observe that by $\vee_2$-free property, any $f\in F$ is below only one $(a_j,b_j)$ and thus \begin{equation}\label{racs}
    F\setminus \{(k,l)\}\subseteq \cup_{j=1}^h[(a_{j-1}+1,b_{j+1}+1),(a_j,b_j)],
\end{equation} where $a_0=b_{h+1}=0$ and $[f_1,f_2]=\{g:f_1\leqslant g \leqslant f_2\}$. We claim that $h\le 2$ and if $h=2$, then $b_1=l$, $a_2=k$. Indeed, if there exists $j$ with $a_j\neq k$, $b_j\neq l$ and there exists $1\le i\le h$, $i\neq j$, then for $(a,b)$ with $a=\max\{a_i,a_j\}$, $b=\max\{b_i,b_j\}$ we have $(a,b)\neq (k,l)$ and thus $(a,b)\notin F$. But $(a,b)$ does not create any strong copy of $\vee_2$ as there is only one element of $F$, namely $(k,l)$, that is larger than $(a,b)$, and also, by (\ref{racs}), any $f\in F$ that is smaller than $(a,b)$ is comparable only to elements in $\{f':f'\leqslant (a,b)\}\cup \{(k,l)\}$.

We distinguish two cases. If there is a unique maximal element $(a,b)$ of $F\setminus \{(k,l)\}$, then adding a maximal chain from $(a,b)$ to $(k,l)$ does not violate the strong $\vee_2$-free property, thus $(a,b)=(k,l-1)$ or $(a,b)=(k-1,l)$. But then $F\setminus \{(k,l)\}$ is strong $\vee_2$-free saturating in $[a]\times [b]$, so by induction we obtain $|F|=(a+b-1)+1=k+l-1$.

Finally, if there are two maximal elements of $F\setminus \{(k,l)\}$, then these must be $(k,j)$ and $(i,l)$ for some $1\le j \le l$ and $1\le i\le k$. As any  $f\in [(i+1,1),(k,j)]$ and $f'\in [(1,j+1),(i,l)]$ form an incomparable pair, by (\ref{racs}), we have that $F\cap [(i+1,1),(k,j)]$ is strong $\vee_2$-saturated in $[(i+1,1),(k,j)]$ and $F\cap [(1,j+1),(i,l)]$ is strong $\vee_2$-saturated in $[(1,j+1),(i,l)]$. Thus by induction, we obtain $|F|=(k-i+j-1)+(i+l-j-1)+1=k+l-1$.
\end{proof}

We finish this section with the proof of Theorem \ref{satvees}, which states the following. Let $P$ be a poset with $dim(P)=2$ such that a strong copy of $P$ in a two dimensional grid cannot contain two neighboring points. Then $sat^*([k]\times [l],P)\ge \max\{k,l\}$.

\begin{proof}[Proof of Theorem \ref{satvees}] Suppose that $F$ is a $P$-saturated subset of $[k] \times [l]$. We prove that it must contain an element in each column, and by an analogous argument for the rows we are done.

We prove by contradiction: suppose there is an empty column. If $F$ is not empty, then we can suppose that there is an empty column next to a non-empty one: $\{(h,j): j \in [l]\} \cap F = \emptyset$ and $\{(g,j): j \in [l]\} \cap F \neq \emptyset$ with either $g=h-1$ or $g=h+1$. Let us define $y$ as $\max\{j:(g,j)\in F\}$ if $g=h-1$ and $\min\{j:(g,j)\in F\}$ if $g=h+1$. Let $a:=(g,y)\in F$ and $b:=(h,y)\notin F$.

Since $F$ is $P$-saturated, $F\cup\{b\}$ contains a strong copy of $P$ that includes $b$. By the property of $P$, the neighboring point $a$ is not in this copy. Also note that all points of the grid, except for certain points in the columns $g$ and $h$, compare the same way to $a$ and $b$ (smaller than, greater than or incomparable to both). By the selection of $a$ and $b$ these exceptional points are not in $F$. Therefore as far as subposets in $F$ are concerned, $a$ and $b$ are interchangeable. This means that there is a strong copy of $P$ with $a$ in the place of $b$. That contradicts the assumption that $F$ is strong $P$-free.
\end{proof}

\section{Open problems}

The widely believed conjecture of forbidden subposet problems for the Boolean case appeared first in \cite{B,GriLu} and considers the limits $\pi_P=\lim_{n\rightarrow \infty}\frac{La(n,P)}{\binom{n}{\lfloor n/2\rfloor}}$ and $\pi^*_P=\lim_{n\rightarrow \infty}\frac{La^*(n,P)}{\binom{n}{\lfloor n/2\rfloor}}$. These limits are yet to be proved to exist, however, the following natural conjecture gives their possible values.

\begin{conjecture}\label{foconj}
For a poset $P$ let us denote by $e(P)$ the largest integer $m$ such that for any $n$, any family $\cF\subseteq 2^{[n]}$ consisting of $m$ \textit{consecutive} levels is weak $P$-free and the parameter $e^*(P)$ is defined analogously for strong $P$-free families. Then $\pi_P=e(P)$ and $\pi^*_P=e^*(P)$ hold.
\end{conjecture}

We conjecture that the values $\pi_P$ and $\pi^*_P$ can be obtained via forbidden subposet problems in the grid for any poset $P$ using Corollary \ref{gridboole}.

\begin{conjecture}\label{mindenszep}
For any at most $d$-dimensional poset $P$ there exist $\pi_{d,P}=\lim_{k\rightarrow \infty}\frac{La([k]^d,P)}{w_{k,d}}$ and $\pi^*_{d,P}=\lim_{k\rightarrow \infty}\frac{La^*([k]^d,P)}{w_{k,d}}$. Moreover, $\lim_{d\rightarrow \infty} \pi_{d,P}=\pi_P$ and $\lim_{d\rightarrow \infty} \pi^*_{d,P}=\pi^*_P$ hold.
\end{conjecture}

Corollary \ref{gridboole} implies $\pi_P\le \pi_{d,P}$ and $\pi^*_P\le \pi^*_{d,P}$ for any $P$ and $d$. The following conjecture states that $\pi_{d,P}$ is monotone decreasing in $d$.

\begin{conjecture}
For any at most $d$-dimensional poset $P$, we have $\pi_{d,P}\ge \pi_{d+1,P}$ and $\pi^*_{d,P}\ge \pi^*_{d+1,P}$
\end{conjecture}

The smallest poset for which Conjecture \ref{foconj} has not been verified is the diamond poset $D_2$. Theorem \ref{doublechain} determines $\pi_{2,D_2}$, and we have the following conjecture for larger values of $d$.

\begin{conjecture}
For any $d\ge 2$, we have $\pi_{d,D_2}=\pi^*_{d,D_2}=2+\frac{1}{d}$.
\end{conjecture}

Concerning the connection of the permutation pattern $J_s$ and $\vee_s$, we state the following conjecture, the first part of which can already be found in \cite{BC}.

\begin{conjecture}\label{satconj}
For any $s\ge 3$, we have $sat(n,n,J_s)=ex(n,n,J_s)=sat^*([n]^2,\vee_s)=La^*([n]^2,\vee_s)$.
\end{conjecture}

We finish with two problems in larger dimensional grids.

\begin{conjecture}
For any $d\ge 3$ and $k_1,k_2,\dots,k_d$ we have $sat^*([k_1]\times [k_2]\times \dots \times [k_d], \vee_2)=\sum_{i=1}^dk_i-1$ and $La^*([k]^d,\vee_2)=(1+\frac{1}{d}+o(1))w_{k,d}$.
\end{conjecture}

\begin{problem}
Determine the possible orders of magnitude of $sat^*([k]^d,P)$. 
\end{problem}

Observe that the dichotomy part of the proof of Theorem \ref{strongsat} stays valid, showing that for any poset $P$, we have that $sat^*([k]^d,P)$ is either constant or at least linear.

\vskip 0.5truecm

\textbf{Acknowledgement.} Research partially sponsored by the National Research, Development and Innovation Office -- NKFIH under the grants K 132696, KH 130371, PD 137779, SNN 129364, FK 132060, and KKP-133819. Research of Vizer was supported by the J\'anos Bolyai Research Fellowship of the Hungarian Academy of Sciences and by the New National Excellence Program under the grant number \'UNKP-21-5-BME-361. Patk\'os was partially supported by the Ministry of Education and Science of the Russian Federation in the framework of MegaGrant no 075-15-2019-1926. Nagy was supported by the the J\'anos Bolyai Research Fellowship of the Hungarian Academy of Sciences.

\end{document}